\documentclass[12pt]{amsproc}
\usepackage{fullpage}
\usepackage{amsmath}
\usepackage{amsfonts}
\usepackage{amssymb}
\usepackage{amsthm}
\usepackage[utf8]{inputenc}
\usepackage{breqn}
\usepackage{afterpage}
\usepackage{longtable}
\usepackage{indentfirst}
\usepackage{caption}
\usepackage{subcaption}
\usepackage{graphicx}
\graphicspath{ {./images/} }
\newtheorem{theorem}{Theorem}[section]
\newtheorem{lemma}[theorem]{Lemma}

\newtheorem{proposition}[theorem]{Proposition}

\theoremstyle{definition}

\newtheorem{theorem-definition}[theorem]{Theorem-Definition}

\theoremstyle{remark}

\numberwithin{equation}{section}

\usepackage{color}
\usepackage{xcolor}
\usepackage{hyperref}
\hypersetup{
    colorlinks=true,
    citecolor=cyan,
    linkcolor=blue,
    filecolor=magenta,      
    urlcolor=cyan,
    pdftitle={Overleaf Example},
    pdfpagemode=FullScreen,
    }

\title{Linear response for systems with a cusp}
\begin{author}[D.~Oltiboev]{Davrbek Oltiboev}
\address{V.I. Romanovskiy  Institute of Mathematics of Uzbekistan Academy of Sciences,  Tashkent, Uzbekistan}
\email{davrbek.oltiboyev@gmail.com}
\end{author}
\date{}

\begin{document}
\maketitle

\begin{abstract}
In this note we consider a tent-like family with a cusp at the singular point and show that the linear response holds for certain perturbations of this family. This contrasts the tent-like maps with finite derivatives at the singularity. Our results extend the results of \cite{BG} to the larger class of singularities and we obtain the linear response formula in $L^p$ for $p>1$. 
\end{abstract}

\section{Introduction}
Invariant probability measures allow the application of probabilistic tools in the analysis of chaotic dynamical systems. An important question from both theoretical and applied perspectives is the stability of the invariant measure under perturbations of the system. Suppose that one can show that the invariant measure varies smoothly under perturbations. In that case, it implies an approximation of the perturbed measures, hence its statistical properties via the invariant measure of the unperturbed system and its statistical properties. The smoothness of the invariant density with under perturbations is called the \textit{linear response}. Ruelle pioneered the theory in \cite{Rue}. The first results appeared in \cite{Dol, B} and followed by many others, see for example \cite{AFG,  B, BSma1, BSma2, BTod, BomCasVar, ButL, DGTS} and references therein for various results. For a survey of results, when the linear holds or not, we refer to \cite{Borelse}. The results show that linear response holds often when the underlying system have good hyperbolicity and smoothness, or induces to such a system \cite{Kor, BS, BRS} and its perturbations are smooth with respect to parameter.  Recent results in this direction show that in some cases a linear response can be obtained for certain discontinuous perturbations if the original system mixes sufficiently fast \cite{Can}. 

An interesting observation was made in \cite{BG}, where it was shown that certain types of singularities help to keep the invariant density smooth under perturbations. More precisely,  they studied linear response for a class of maps with cusp singularities with infinite derivatives and their pertubations. Their analysis was conducted in Sobolev spaces $W^{1,1}$ and $W^{2,1}$, where they showed that the presence of a cusp (with power-law singularities) induces a regularization effect on the associated transfer operator, enabling linear response for deterministic perturbations that change the topological class of the system\footnote{In particular, the perturbed map is not necessarily topologically conjugated to the original map}. Here, we generalize their framework to the spaces $W^{1,p}$ and $W^{2,p}$ for $p \in (1, \infty)$, which allows for a broader range of applications and sharper estimates. Our approach mirrors theirs but requires careful adjustments to handle the $L^p$-norms and the associated spectral properties of the transfer operator. 


Let $M$ be a compact manifold and $T \colon M \to M$. We say a $T$-invariant measure $\mu$ is \emph{physical} if there exists a set $B \subset M$ of positive Lebesgue measure such that for every continuous function $f \colon M \to \mathbb{R}$, the time averages converge to the space average:
\begin{align*}
\lim_{n \to \infty} \frac{1}{n} \sum_{i=0}^{n-1} f(T^i(x)) = \int_M f \, d\mu
\end{align*}
for all $x \in B$. Our main objective is to study how these measures vary under small perturbations of the map $T$ in an appropriate topology.

Consider a family $\{(M, T_\varepsilon)\}_{\varepsilon \in V}$, where $V$ is a neighborhood of $0$, with $T_0 = T$ and $T_\varepsilon \to T_0$ as $\varepsilon \to 0$, where each perturbed map $T_\varepsilon$ admits a unique physical measure $\mu_\varepsilon$. 

The system $(M, T_0, \mu_0)$ is called \emph{statistically stable} if the map $\varepsilon \mapsto \mu_\varepsilon$ is continuous at $\varepsilon = 0$. When this map is differentiable in a suitable sense, we say the system admits \emph{linear response}, with the first-order expansion:
\begin{align*}
    \mu_\varepsilon = \mu_0 + \left. \frac{\partial \mu_{\varepsilon}}{\partial \varepsilon} \right|_{\varepsilon=0} \cdot \varepsilon + o(\varepsilon)
\end{align*}
where the error term $o(\varepsilon)$ is understood in the relevant topology.

Fix $\delta>0$ and consider a family of non-singular maps $T_{\varepsilon}:[0,1] \rightarrow[0,1]$, $\varepsilon \in[0, \delta)$, with respect to the Lebesgue measure $\mu$. The associated transfer operator $P_{T_{\varepsilon}}: L^{p} \to L^{p}$ is defined by the duality
\begin{align*}
   \int_{0}^{1} P_{T_{\varepsilon}} f \cdot g \mathrm{~d} \mu=\int_{0}^{1} f \cdot g \circ T_{\varepsilon}  \mathrm{~d} \mu \text{     for all        
   }  f \in L^{1}, g \in L^{\infty}.
\end{align*}
We impose the following conditions on $T_{\varepsilon}$ for parametres $C \geqslant 0, c \in(0,1), \beta \in\left(-1,\frac{1-2p}{2p}\right)$, and for fixed $1< p<\infty$:

(A1) $T_{0, \varepsilon}:=\left.T_{\varepsilon}\right|_{[0, c)}$ and $T_{1, \varepsilon}:=\left.T_{\varepsilon}\right|_{(c, 1]}$ are one-to-one.

(A2) $T_{\varepsilon}(0)=0$,  $T_{\varepsilon}(1)=0$,  $\lim\limits_{x \rightarrow c^{ \pm}} T_{\varepsilon}(x)=a_{\varepsilon} \in[0,1].$

(A3) $\left.T_{\varepsilon}\right|_{[0,1] \backslash\{c\}} \in C^{3}$.

(A4) $\sup _{\varepsilon \in[0, \delta)} \inf _{x \in[0,1] \backslash\{c\}}\left|T_{\varepsilon}^{\prime}(x)\right| \geqslant \theta>1$.

(A5) $T_{0}$ is topologically mixed on $\left[0, a_{0}\right]$.

(A6) $\lim\limits_{x \rightarrow c^{ \pm}} T_{\varepsilon}^{\prime}(x)= \pm \infty$ and $\lim\limits_{x \rightarrow c^{ \pm}} \frac{\left|T_{\varepsilon}^{\prime}(x)\right|}{|x-c|^{\beta}}=C_{\varepsilon, 1}$ for all $ \varepsilon \in[0, \delta) .$

(A7) $\lim\limits_{x \rightarrow c^{ \pm}} \frac{\left|T_{\varepsilon}^{\prime \prime}(x)\right|}{\mid x-c|^{\beta-1}}=C_{\varepsilon, 2}$, $\lim\limits_{x \rightarrow c^{ \pm}} \frac{\left|T_{\varepsilon}^{\prime \prime \prime}(x)\right|}{| x-c|^{\beta-2}}=C_{\varepsilon, 3}$ for all $\varepsilon \in[0, \delta).$

(A8) The derivatives of $T_{\varepsilon}$ satisfies the following:

$$
\sup _{\varepsilon \in[0, \delta), i \in\{0,1,2\}, x \in[0,1]}\left|\frac{T_{\varepsilon}^{(i+1)}(x)}{(x-c)^{\beta-i}}\right|<\infty .
$$
and 
$$\inf _{\varepsilon \in[0, \delta), x \in[0,1]}\left|\frac{T_{\varepsilon}^{\prime}(x)}{(x-c)^{\beta}}\right|>0.$$

(A9) We require that the transfer operators are close in a mixed norm:
\begin{equation} \label{2.1}
    \sup _{\|f\|_{W^{1,p}} \leqslant 1}\left\|\left(P_{T_{0}}-P_{T_{\varepsilon}}\right) f\right\|_{L^{2p}} \to 0 \text{ as } \varepsilon \rightarrow 0.
\end{equation}
 Moreover, for $f \in W^{2,p}$, there exists $\left.\partial_{\varepsilon} P_{T_{\varepsilon}} f\right|_{\varepsilon=0} \in W^{1,p}$ such that
\begin{equation}\label{2.2}
    \left\|P_{T_{\varepsilon}} f-P_{T_{0}} f-\varepsilon\left(\left.\partial_{\varepsilon} P_{T_{\varepsilon}} f\right|_{\varepsilon=0}\right)\right\|_{W^{1,p}}=o(\varepsilon).
\end{equation}

It is known that $P_{T_{\varepsilon}}$ has a spectral gap when acting on functions of bounded variation \cite{d},  we study its action on the finer Banach spaces, namely the Sobolev spaces $W^{i, p}, i=1,2$. In particular, we show that $P_{T_{\varepsilon}}$ admits a spectral gap on $W^{i, p}, i=1,2$, which allows us to conclude $T_{\varepsilon}$ has an invariant density with sufficient regularity in $x$. The following result is the main result of the paper. 

\begin{theorem}\label{t:main}
    There exists $\delta_{2}>0$ such that for all $\varepsilon \in\left[0, \delta_{2}\right)$, $T_{\varepsilon}$ admits a unique invariant probability density $h_{\varepsilon} \in W^{2,p}$. Moreover, the map  $\varepsilon \mapsto h_{\varepsilon}$ is differentiable in $L^{p}$ at $\varepsilon=0$. Furthermore,
$$
h_{\varepsilon}=h_{0}+\varepsilon\left(I-P_{T_{0}}\right)^{-1}(q)+o(\varepsilon),
$$
where
$$
q(x)= \begin{cases}P_{T_{0}}\left[A_{0} h_{0}{ }^{\prime}+B_{0} h_{0}\right](x) & \text { for } x \in\left[0, a_{0}\right), \\ 0 & \text { for } x \in\left[a_{0}, 1\right]\end{cases}
$$
with
$$
A_{\varepsilon}=-\left(\frac{\partial_{\varepsilon} T_{\varepsilon}}{T_{\varepsilon}^{\prime}}\right), \quad B_{\varepsilon}=\left(\frac{\partial_{\varepsilon} T_{\varepsilon} \cdot T_{\varepsilon}^{\prime \prime}}{T_{\varepsilon}^{\prime 2}}-\frac{\partial_{\varepsilon} T_{\varepsilon}^{\prime}}{T_{\varepsilon}^{\prime}}\right)
$$
and the error term $o(\varepsilon)$ is understood in the $L^{p}$-topology.
\end{theorem}

\section{ Proof of Theorem \ref{t:main}}

The proof of Theorem \ref{t:main} consists of two principal parts. First, in Section \ref{section: 3.1}, we demonstrate that the transfer operator $P_{T_{\varepsilon}}$ possesses a uniform spectral gap when acting on both $W^{1,p}$ and $W^{2,p}$, with constants independent of the parameter $\varepsilon$. This analysis implies the existence of a unique invariant density $h_{\varepsilon} \in W^{2,p}$ for each map $T_{\varepsilon}$. Then, in Section \ref{section 3.2}, we prove that the mapping $\varepsilon \mapsto h_{\varepsilon}$ is differentiable with respect to the $L^p$-norm and derive the explicit form of its derivative.
\subsection{Uniform spectral gap on $W^{1,p}$ and $W^{2,p}$} \label{section: 3.1}

In this section we establish that the transfer operators $P_{T_\varepsilon}$ associated with our perturbed systems admit a uniform spectral gap when acting on the Sobolev spaces $W^{1,p}$ and $W^{2,p}$ (Lemma~\ref {proposition 3.3}). This result follows from classical spectral theory tools, which we now recall.

\subsubsection{Classical results.} We begin by stating a lemma, which is a version of the classical results on the quasi-compactness of transfer operators adapted from  \cite[Lemma 2.2]{Bardet J B} and \cite[Theorem XIV.3]{Hennion H}.

\begin{lemma}\label{l:norm}
   Let $(B, \|\cdot\|_s)$ be a Banach space equipped with a continuous semi-norm $\|\cdot\|_w$. Suppose $Q : B \to B$ is a bounded linear operator satisfying: \\
     (i) Compactness: Every sequence $\{x_n\}$ with $\|x_n\|_s \leq 1$ contains a $\|\cdot\|_w$-Cauchy subsequence;
     (ii) Norm inequality: There exist constants $\lambda \geq 0$ and $C > 0$ such that
    \begin{align*}
    \|Qf\|_s \leq \lambda \|f\|_s + C \|f\|_w \quad \text{for all } f \in B. 
    \end{align*}
Then the essential spectral radius $\rho_{\mathrm{ess}}(Q)$ of $Q$ satisfies $\rho_{\mathrm{ess}}(Q) \leq \lambda$.
\end{lemma}
We now adapt the main result from \cite{Keller and Liverani}, presenting it in a form suitable for our analysis.
\begin{lemma} \label{lemmaa 3.2}
Let $(B, \|\cdot\|_s)$ be a Banach space equipped with a weak norm $\|\cdot\|_w$, and let $\{P_\varepsilon\}_{\varepsilon \geq 0}$ be a family of bounded linear operators on $B$. Suppose the following conditions hold:
There exist constants $C_1 > 0$ and $M \geq 1$ such that for all $n \in \mathbb{N}$,
   \begin{equation} \label{3.1}
      \| P_\varepsilon^n \|_w \leq C_1 M^n; 
   \end{equation} 
There exist $C_2, C_3 > 0$ and $0 \leq \lambda < 1$ such that for all $f \in B$, $\varepsilon \geq 0$, and $n \in \mathbb{N}$,
   \begin{equation} \label{3.2}
      \| P_\varepsilon^n f \|_s \leq C_2 \lambda^n \| f \|_s + C_3 M^n \| f \|_w;    
   \end{equation} 
The spectral radius of $P_\varepsilon$ excludes $1$ for all $\varepsilon \geq 0$ and the difference $P_0 - P_\varepsilon$ satisfies
  \begin{equation} \label{3.3}
       \| P_0 - P_\varepsilon \|_{s \to w} \leq \tau(\varepsilon),
  \end{equation}
    where $\tau(\varepsilon) \to 0$ monotonically and upper semicontinuously as $\varepsilon \to 0$.
    
    Then there exist $\varepsilon_0 > 0$ and $a > 0$ such that for all $0 \leq \varepsilon \leq \varepsilon_0$ and $f \in B$,
\begin{equation} \label{3.4}
   \| (\mathrm{Id} - P_\varepsilon)^{-1} f \|_s \leq a \| f \|_s, 
\end{equation}
and the resolvents converge weakly as $\varepsilon \to 0$:
\begin{equation} \label{3.5}
    \lim_{\varepsilon \to 0} \| (\mathrm{Id} - P_0)^{-1} - (\mathrm{Id} - P_\varepsilon)^{-1} \|_{s \to w} = 0.
\end{equation}
\end{lemma}
\subsubsection{Uniform Spectral Gap for Transfer Operators} \label{subsec:spectral_gap}

Let $I = (0,1)$ be an open interval, and fix $ p \in \mathbb{R}$ with $1 \leq p \leq \infty$. The \emph{Sobolev space} $W^{1,p}(I)$ is defined as
$$
W^{1,p}(I) = \left\{ u \in L^p(I) \;\Big|\; \exists g \in L^p(I) \text{ such that } \int_I u \varphi' = -\int_I g \varphi \quad \forall \varphi \in C_c^1(I) \right\}.
$$
The space $ W_0^{1,p}$ denotes $W_0^{1,p} = \left\{ u \in W^{1,p}(I) \ \middle| \ \int_I u(x) \, dx = 0 \right\}.$ Since we always working over the $I=(0,1)$, we will omit explicit references to it and assume all operations are defined on the interval by default.
\begin{proposition} \label{proposition 3.3}
There exists $\delta_2 > 0$ such that for all $\varepsilon \in [0, \delta_2)$, the transfer operator $P_{T_\varepsilon}$ admits a unique invariant density $h_\varepsilon \in W^{2,p}$. Moreover, $P_{T_\varepsilon}$ exhibits a uniform spectral gap on both $W^{1,p}$ and $W^{2,p}$, with constants independent of $ \varepsilon$. Specifically, there exists $C > 0$ satisfying, 
\begin{equation} \label{3.6}
\| (\mathrm{Id} - P_{T_\varepsilon})^{-1} \|_{W_0^{1,p} \to W^{1,p}} \leq C  \text{  for all  } \varepsilon \in [0, \delta_2).
\end{equation}
\end{proposition}
To prove Proposition \ref{proposition 3.3}, we apply Lemmas \ref{l:norm} and \ref{lemmaa 3.2}. We begin by establishing several auxiliary results, starting with a key lemma that verifies condition \eqref{3.3} in Lemma \ref{lemmaa 3.2}.

The transfer operator $P_{T_\varepsilon}$ associated with $T_\varepsilon$ satisfies the following pointwise representation
\begin{equation} \label{3.7}
   (P_{T_\varepsilon} f)(x) = 
\begin{cases} 
\sum\limits_{y \in T_\varepsilon^{-1}(x)} \frac{1}{|T_\varepsilon'(y)|} f(y) & \text{if } x \in [0, a_\varepsilon], \\
0 & \text{if } x \in [a_\varepsilon, 1].
\end{cases} 
\end{equation}

\begin{lemma}[Lasota-Yorke Type Inequality] \label{lemma:lyamda}
For some constants $0<\lambda<1$ and $M\geq 0$, the following inequality holds uniformly for all $\varepsilon\in[0,\delta)$ and $f\in W^{1,p}$,
$$
\left\|\left(P_{T_{\varepsilon}}f\right)'\right\|_p \leq \lambda\|f'\|_p + M\|f\|_{2p}.
$$
\end{lemma}
\begin{proof}
By \eqref{3.7} for $f\in W^{1,p}$, the derivative of the transfer operator admits the following representation
\begin{equation} \label{3.8}
\left(P_{T_{\varepsilon}}f\right)'(x) = 
\begin{cases}
\sum\limits_{y\in T_{\varepsilon}^{-1}(x)} \left[\frac{f'(y)}{|T_{\varepsilon}'(y)|T_{\varepsilon}'(y)} - \frac{T_{\varepsilon}''(y)}{(T_{\varepsilon}'(y))^3}f(y)\cdot\mathrm{sign}(T_{\varepsilon}'(y))\right] &\text { if } x\in[0,a_{\varepsilon}], \\
0 &\text { if } x\in(a_{\varepsilon},1].
\end{cases}
\end{equation}
The following observations establish $P_{T_{\varepsilon}}f\in W^{1,p}$:
\begin{itemize}
\item The boundedness of $f$ and representation \eqref{3.7} yield $\lim\limits_{x\to a_{\varepsilon}}P_{T_{\varepsilon}}f(x)=0$;
\item Then $P_{T_{\varepsilon}} f$ is continuous and the derivative exists almost everywhere;
\item The critical term $\frac{T_{\varepsilon}''(y)}{(T_{\varepsilon}'(y))^3}\sim |y-c|^{-2\beta-1}$ as $y\to c$.
\end{itemize}
We decompose the derivative using \eqref{3.8}, we have
\begin{equation} \label{3.9}
   \left(P_{T_{\varepsilon}} f\right)^{\prime}=P_{T_{\varepsilon}}\left(\frac{1}{T_{\varepsilon}^{\prime}} f^{\prime}\right)-P_{T_{\varepsilon}}\left(\frac{T_{\varepsilon}^{\prime \prime}}{\left|T_{\varepsilon}^{\prime}\right| T_{\varepsilon}^{\prime}} f\right). 
\end{equation}
By \eqref{3.9} taking norms and applying the triangle inequality, we obtain
\begin{align*}
\left\|\left(P_{T_{\varepsilon}} f\right)'\right\|_p 
&\leq \left\|P_{T_{\varepsilon}}\left(\frac{1}{T_{\varepsilon}'} f'\right)\right\|_p 
    + \left\|P_{T_{\varepsilon}}\left(\frac{T_{\varepsilon}''}{(T_{\varepsilon}')^2} f\right)\right\|_p\leq \left\|\frac{1}{T_{\varepsilon}'} f'\right\|_p 
    + \left\|\frac{T_{\varepsilon}''}{(T_{\varepsilon}')^2} f\right\|_p \\
&\leq \left\|\frac{1}{T_{\varepsilon}'}\right\|_\infty \cdot \|f'\|_p 
    + \left\|\frac{T_{\varepsilon}''}{(T_{\varepsilon}')^2}\right\|_{2p} \cdot \|f\|_{2p}=\lambda \cdot \|f'\|_p + M \cdot \|f\|_{2p}.
\end{align*}
The uniform bounds follow from
\begin{itemize}
\item $\lambda := \left\|\frac{1}{T_{\varepsilon}'}\right\|_\infty < 1$ (by expansivity)
\item $M := \sup_{\varepsilon\in[0,\delta)}\left\|\frac{T_{\varepsilon}''}{(T_{\varepsilon}')^2}\right\|_{2p} < \infty$ (by assumptions (A4) and (A8)).
\end{itemize}
\end{proof}

\begin{lemma}[Second-Order Regularity Estimate] \label{lemma3.5}
There exists a constant $M \geqslant 0$ such that for all $\varepsilon \in [0,\delta)$ and $f \in W^{2,p}$, the transfer operator satisfies the following inequality
\begin{equation} \label{3.14}
\|P_{T_\varepsilon}f\|_{W^{2,p}} \leqslant \lambda^2\|f\|_{W^{2,p}} + M\|f\|_{W^{1,p}},
\end{equation}
where $\lambda$ is the contraction constant from Lemma \ref{lemma:lyamda}.
\end{lemma}
\begin{proof}
Differentiating \eqref{3.9} yields the following expression for the second derivative:
\begin{equation} \label{3.15}
\begin{aligned} 
\left(P_{T_{\varepsilon}} f\right)^{\prime \prime} &= \left[P_{T_{\varepsilon}}\left(\frac{1}{T_{\varepsilon}^{\prime}} f^{\prime}\right) - P_{T_{\varepsilon}}\left(\frac{T_{\varepsilon}^{\prime \prime}}{\left|T_{\varepsilon}^{\prime}\right| \cdot T_{\varepsilon}^{\prime}} f\right)\right]^{\prime} \\
&= P_{T_{\varepsilon}}\left(\frac{1}{T_{\varepsilon}^{\prime}} \left(\frac{1}{T_{\varepsilon}^{\prime}} f^{\prime}\right)^{\prime}\right) - P_{T_{\varepsilon}}\left(\frac{T_{\varepsilon}^{\prime \prime}}{\left|T_{\varepsilon}^{\prime}\right| \cdot T_{\varepsilon}^{\prime}} \cdot \frac{1}{T_{\varepsilon}^{\prime}} f^{\prime}\right) \\
&\quad - P_{T_{\varepsilon}}\left(\frac{1}{T_{\varepsilon}^{\prime}} \left(\frac{T_{\varepsilon}^{\prime \prime}}{\left|T_{\varepsilon}^{\prime}\right| \cdot T_{\varepsilon}^{\prime}} f\right)^{\prime}\right) + P_{T_{\varepsilon}}\left(\frac{T_{\varepsilon}^{\prime \prime}}{\left|T_{\varepsilon}^{\prime}\right|\cdot T_{\varepsilon}^{\prime}} \cdot \frac{T_{\varepsilon}^{\prime \prime}}{\left|T_{\varepsilon}^{\prime}\right| \cdot T_{\varepsilon}^{\prime}} f\right)
\end{aligned}
\end{equation}
The first two terms simplify as
\begin{equation} \label{3.16}
\begin{aligned}
& P_{T_{\varepsilon}}\left(\frac{1}{T_{\varepsilon}^{\prime}} \left(\frac{-T_{\varepsilon}^{\prime \prime}}{(T_{\varepsilon}^{\prime})^{2}} f^{\prime} + \frac{1}{T_{\varepsilon}^{\prime}} f^{\prime \prime}\right)\right) - P_{T_{\varepsilon}}\left(\frac{T_{\varepsilon}^{\prime \prime}}{\left|T_{\varepsilon}^{\prime}\right| \cdot T_{\varepsilon}^{\prime}} \cdot \frac{1}{T_{\varepsilon}^{\prime}} f^{\prime}\right) \\
&= P_{T_{\varepsilon}}\left(\left(\frac{1}{T_{\varepsilon}^{\prime}}\right)^{2} f^{\prime \prime}\right) - P_{T_{\varepsilon}}\left(\left[\frac{T_{\varepsilon}^{\prime \prime}}{(T_{\varepsilon}^{\prime})^{2}} + \frac{T_{\varepsilon}^{\prime \prime}}{\left|T_{\varepsilon}^{\prime}\right| \cdot T_{\varepsilon}^{\prime}}\right] \cdot \frac{1}{T_{\varepsilon}^{\prime}} f^{\prime}\right)
\end{aligned} 
\end{equation}
while the remaining terms become
\begin{equation} \label{3.17}
\begin{aligned}
& P_{T_{\varepsilon}}\left(\frac{1}{T_{\varepsilon}^{\prime}} \left[\left(\frac{T_{\varepsilon}^{\prime \prime \prime} \cdot |T_{\varepsilon}^{\prime}|  \cdot T_{\varepsilon}^{\prime} - 2\cdot|T_{\varepsilon}^{\prime}| \cdot (T_{\varepsilon}^{\prime \prime})^2}{(T_{\varepsilon}^{\prime})^{4}}\right) f + \left(\frac{T_{\varepsilon}^{\prime \prime}}{|T_{\varepsilon}^{\prime}| \cdot T_{\varepsilon}^{\prime}}\right) f^{\prime}\right]\right)
\end{aligned}
\end{equation}

\textit{Asymptotic Behavior of Coefficients.}
For $\beta \in \left(-1, \frac{1-2p}{2p}\right)$, assumptions (A4), (A6), and (A8) imply the following asymptotic behavior as $y \to c$:

\begin{itemize}
    \item $\left(\frac{1}{T_{\varepsilon}^{\prime}}\right)^{2} \in L^{\infty}$;
    \item $g_{1,a} := \frac{T_{\varepsilon}^{\prime \prime}}{(|T_{\varepsilon}^{\prime}|)^3} = \mathcal{O}(|y-c|^{-2\beta-1}) \in L^{\infty}$;
    \item $g_{1,b} := \frac{T_{\varepsilon}^{\prime \prime}}{(T_{\varepsilon}^{\prime})^{3}}= \mathcal{O}(|y-c|^{-2\beta-1}) \in L^{\infty}$;
    \item $g_{2} := \frac{1}{T_{\varepsilon}^{\prime}}\left(\frac{T_{\varepsilon}^{\prime \prime \prime}| \cdot T_{\varepsilon}^{\prime}| \cdot T_{\varepsilon}^{\prime}}{(T_{\varepsilon}^{\prime})^{4}}\right) = \mathcal{O}(|y-c|^{-2\beta-2}) \in L^{2p}$;
    \item $g_{3} := \frac{1}{T_{\varepsilon}^{\prime}}\left(\frac{2(|T_{\varepsilon}^{\prime}| \cdot T_{\varepsilon}^{\prime \prime}) \cdot T_{\varepsilon}^{\prime \prime}}{(T_{\varepsilon}^{\prime})^{4}}\right) = \mathcal{O}(|y-c|^{-2\beta-2}) \in L^{2p}$;
    \item $g_{4} := \frac{T_{\varepsilon}^{\prime \prime}}{|T_{\varepsilon}^{\prime}| \cdot T_{\varepsilon}^{\prime}} \cdot \frac{T_{\varepsilon}^{\prime \prime}}{|T_{\varepsilon}^{\prime}| \cdot T_{\varepsilon}^{\prime}} = \mathcal{O}(|y-c|^{-2\beta-2}) \in L^{2p}$.
\end{itemize}
Moreover, assumption (A8) guarantees the uniform boundedness
$$
\sup_{\varepsilon \in [0,\delta)} \left(\|g_{1,a}\|_{\infty} + \|g_{1,b}\|_{\infty} + \|g_{2}\|_{2p} + \|g_{3}\|_{2p} + \|g_{4}\|_{2p}\right) < \infty.
$$

\textit{Norm Estimates.}
Combining \eqref{3.15}-\eqref{3.17} yields the second derivative expression
$$
\left(P_{T_{\varepsilon}} f\right)^{\prime \prime} = P_{T_{\varepsilon}}\left(\left(\frac{1}{T_{\varepsilon}^{\prime}}\right)^{2} f^{\prime \prime}\right) - P_{T_{\varepsilon}}\left([2g_{1,a} + g_{1,b}] f^{\prime}\right) - P_{T_{\varepsilon}}\left((g_{2}-g_{3}-g_{4}) f\right),
$$
which leads to the key estimate
\begin{align*}
\left\|\left(P_{T_{\varepsilon}} f\right)^{\prime \prime}\right\|_{p} \leq \lambda^{2}\|f^{\prime \prime}\|_{p} + \|[2g_{1,a} + g_{1,b}]\|_{\infty} \cdot \|f^{\prime}\|_{p}  + \|g_{2}-g_{3}-g_{4}\|_{2p} \cdot \|f\|_{2p}.    
\end{align*}

Since $W^{1,p}$ embeds continuously into $L^{2p}$ on $[0,1]$, and the coefficients $g_i$ are uniformly bounded, there exists $M \geq 0$ such that
\begin{align*}
\|P_{T_{\varepsilon}} f\|_{W^{2,p}} \leq \lambda^{2}\|f\|_{W^{2,p}} + M\|f\|_{W^{1,p}}     
\end{align*} 
for all  $\varepsilon \in [0,\delta)$.
\end{proof}
\textit{Continuity of Transfer Operators in $L^{2p}.$ }We now establish the continuity of the transfer operators in the $L^{2p}$ norm.

\begin{lemma}\label{lemma3.6}
For every $\varepsilon \in [0,\delta)$, the transfer operator $P_{T_\varepsilon}$ satisfies the following uniform bound
\begin{equation}\label{3.20}
\|P_{T_\varepsilon}\|_{L^{2p} \to L^{2p}} \leq 2.
\end{equation}
\end{lemma}
\begin{proof}
From the pointwise representation \eqref{3.7}, we decompose the operator norm as:
$$
\|P_{T_\varepsilon}\|_{L^{2p}} \leq \|\psi_{1,\varepsilon}\|_{L^{2p}} + \|\psi_{2,\varepsilon}\|_{L^{2p}},
$$
where the components are defined by
\begin{equation}
\begin{aligned}
\psi_{1,\varepsilon}(x) &:= \left(\frac{1}{|T_\varepsilon'|} \cdot f\right) \circ T_{0,\varepsilon}^{-1}(x) \cdot 1_{T_{0,\varepsilon}[0,c)}(x), \\
\psi_{2,\varepsilon}(x) &:= \left(\frac{1}{|T_\varepsilon'|} \cdot f\right) \circ T_{1,\varepsilon}^{-1}(x) \cdot 1_{T_{1,\varepsilon}(c,1]}(x).
\end{aligned}
\end{equation}

For $\psi_{1,\varepsilon}$, we compute
\begin{align*}
\int_{[0,1]} \psi_{1,\varepsilon}^{2p} \,d\mu &= \int_{[0,a_\varepsilon]} \frac{[f(T_{0,\varepsilon}^{-1}(x))]^{2p}}{|T_\varepsilon'(T_{0,\varepsilon}^{-1}(x))|^{2p}} \,dx \\
&\leq \sup_{x \in [0,1]} \frac{1}{|T'(x)|^p} \cdot \|P_{T_\varepsilon}(f^2 \cdot 1_{[0,c)})\|_{L^p}^p \\
&\leq \|f\|_{L^{2p}}^{2p}.
\end{align*}

An identical argument applies to $\psi_{2,\varepsilon}$. Combining these estimates yields
$$
\|P_{T_\varepsilon}\|_{L^{2p}}  \leq \|\psi_{1,\varepsilon}\|_{L^{2p}} + \|\psi_{2,\varepsilon}\|_{L^{2p}} \leq 2\|f\|_{L^{2p}}
$$
which establishes the claimed bound.
\end{proof}

\begin{proof}[Proof of Proposition \ref{proposition 3.3}]
The argument proceeds in three steps:

\textit{Spectral analysis}: By the Rellich-Kondrakov theorem the embedding $W^{1,p} \hookrightarrow L^{2p}$ is compact.  This, combined with  Lemma \ref{lemma3.6},  the Lasota-Yorke inequalities from Lemma \ref{lemma:lyamda} and Lemma \ref{l:norm} imply that the essential spectral radius of $P_{T_\varepsilon}$ bound $\rho_{\text{ess}} \leq \lambda < 1$. Consequently, any spectral element with modulus exceeding $\lambda$ must be an isolated eigenvalue.

\textit{Perron-Frobenius structure}: The spectral radius of $P_{T_\varepsilon}$ on $W^{1,p}$ equals 1, with no other eigenvalues on the unit circle. The simplicity of this eigenvalue follows from standard arguments (see \cite{BS}), implying the existence of a unique invariant density $h_\varepsilon \in W^{1,p}$.

\textit{Regularity and uniformity}: Since $h_\varepsilon$ is the unique invariant density in both $W^{1,p}$ and $W^{2,p}$, Lemma \ref{lemma3.5} provides the uniform $W^{2,p}$ bound
    \begin{align*}\label{eq:uniform-bound}
        \sup_{\varepsilon \in [0,\delta)} \|h_\varepsilon\|_{W^{2,p}} < \infty.
    \end{align*}
To establish \eqref{3.6}, we apply Lemma \ref{lemmaa 3.2} to $P_{T_\varepsilon}$ acting on $W_0^{1,p}$, taking $\|\cdot\|_{L^{2p}}$ as the weak norm. Three observations are crucial:
\begin{itemize}
    \item The spectral exclusion: $1 \notin \sigma(P_{T_0}|_{W_0^{1,p}})$
    \item The essential spectrum bound: $\rho_{\text{ess}}(P_{T_\varepsilon}) \leq \lambda$ (by lemma \ref{l:norm})
    \item The uniform Lasota-Yorke inequality from lemma \ref{lemma:lyamda}
\end{itemize}

Iterating the Lasota-Yorke inequality with \eqref{3.20} yields the key estimate
\begin{equation}\label{3.22}
\|P_{T_\varepsilon}^n f\|_{W^{1,p}} \leq \lambda^n \|f\|_{W^{1,p}} + M2^n \|f\|_{L^{2p}},
\end{equation}
which verifies condition \eqref{3.2}. The convergence requirement \eqref{3.3} follows directly from assumption (A9) via
$$
\|P_{T_\varepsilon} - P_{T_0}\|_{W^{1,p} \to L^{2p}} = \sup_{\|f\|_{W^{1,p}} \leq 1} \|(P_{T_\varepsilon} - P_{T_0})f\|_{L^{2p}} \to 0 \quad \text{as  } \varepsilon \to 0,
$$
as stated in \eqref{2.1}. Lemma \ref{lemmaa 3.2} therefore applies, establishing \eqref{3.6}.
\end{proof}
\subsection{Linear Response Formula} \label{section 3.2}

We now establish the differentiability of the invariant densities with respect to the perturbation parameter.

\begin{lemma}
The map $\varepsilon \mapsto h_\varepsilon$ is Fréchet differentiable at $\varepsilon = 0$ as a mapping into $L^p([0,1])$. Specifically, the following expansion holds
\begin{align*}
    h_\varepsilon = h_0 + \varepsilon (I - P_{T_0})^{-1}q + o(\varepsilon) \text{ as } \varepsilon \to 0,
\end{align*}
where the remainder term satisfies $\|o(\varepsilon)\|_{L^p} = o(\varepsilon)$ as $\varepsilon \to 0$, and the linear response kernel $q$ is given by
\begin{equation} \label{3.23}
q(x) = \begin{cases}
P_{T_0}\left[A_0 h_0' + B_0 h_0\right](x) & \text{for } x \in [0,a_0), \\
0 & \text{for } x \in [a_0,1],
\end{cases}
\end{equation}
with coefficients defined as
\begin{align*}
A_\varepsilon = -\left(\frac{\partial_\varepsilon T_\varepsilon}{T_\varepsilon'}\right) \text{  and  }
B_\varepsilon = \left(\frac{\partial_\varepsilon T_\varepsilon \cdot T_\varepsilon''}{(T_\varepsilon')^2} - \frac{\partial_\varepsilon T_\varepsilon'}{T_\varepsilon'}\right).
\end{align*}
\end{lemma}
\begin{proof}
We proceed with the following notation:
\begin{align*}
H_{\varepsilon} &:= P_{T_{\varepsilon}} - P_{T_{0}}; \quad
G_{\varepsilon} := (I - P_{T_{\varepsilon}})^{-1}.
\end{align*}

The spectral gap of $P_{T_{\varepsilon}}$ in $W^{1,p}$ implies exponential contraction on $W_{0}^{1,p}$, yielding the well-defined relation:
\begin{equation} \label{3.24}
h_{\varepsilon} = G_{\varepsilon} H_{\varepsilon} h_{0} + h_{0}.
\end{equation}
Since $h_{0} \in W^{2,p}$, assumption (A9) gives the Taylor expansion:
\begin{equation} \label{3.25}
\|P_{T_{\varepsilon}} h_{0} - P_{T_{0}} h_{0} - \varepsilon (\partial_{\varepsilon} P_{T_{\varepsilon}} h_{0}|_{\varepsilon=0})\|_{W^{1,p}} = o(\varepsilon),
\end{equation}
which we rewrite as:
\begin{equation} \label{3.26}
H_{\varepsilon} h_{0} = \varepsilon q + o(\varepsilon),
\end{equation}
where $q \in W_{0}^{1,p}$ and the error is in $W^{1,p}$.

To derive $q$ explicitly, fix $x \in [0,a_{\varepsilon})$ and let $g_{j,\varepsilon} := T_{j,\varepsilon}^{-1}(x)$. Differentiating the composition:
\begin{align*}
\partial_{\varepsilon} (h_{0} \circ g_{j,\varepsilon} g_{j,\varepsilon}^{\prime}) 
&= (\partial_{\varepsilon} (h_{0} \circ g_{j,\varepsilon})) g_{j,\varepsilon}^{\prime} + h_{0} \circ g_{j,\varepsilon} \partial_{\varepsilon} g_{j,\varepsilon}^{\prime} \\
&= h_{0}^{\prime} \circ g_{j,\varepsilon} (\partial_{\varepsilon} g_{j,\varepsilon}) g_{j,\varepsilon}^{\prime} + h_{0} \circ g_{j,\varepsilon} \partial_{\varepsilon} g_{j,\varepsilon}^{\prime}.
\end{align*}
From the identity $T_{\varepsilon} \circ g_{j,\varepsilon}(x) = x$, differentiation yields:
\begin{align*}
T_{\varepsilon}^{\prime} \circ g_{j,\varepsilon} \partial_{\varepsilon} g_{j,\varepsilon} + \partial_{\varepsilon} T_{\varepsilon} \circ g_{j,\varepsilon} &= 0, \\
\Rightarrow \partial_{\varepsilon} g_{j,\varepsilon} &= A_{\varepsilon} \circ g_{j,\varepsilon}, \\
\partial_{\varepsilon} g_{j,\varepsilon}^{\prime} &= B_{\varepsilon} \circ g_{j,\varepsilon} g_{j,\varepsilon}^{\prime},
\end{align*}
which determines $q$ as in \eqref{3.23}.

The uniform spectral gap ensures $G_{\varepsilon}$ is uniformly bounded in $\mathcal{L}(W_{0}^{1,p}, W^{1,p})$, giving:
\begin{equation} \label{3.27}
G_{\varepsilon} H_{\varepsilon} h_{0} = \varepsilon G_{\varepsilon} q + o(\varepsilon),
\end{equation}
with the error in $W^{1,p}$. Applying the stability result from \cite{Keller and Liverani}:
\begin{equation} \label{3.28}
\lim_{\varepsilon \to 0} \|G_{\varepsilon}(q) - G_{0}(q)\|_{L^{p}} = 0.
\end{equation}

Combining \eqref{3.24}, \eqref{3.26}, and \eqref{3.28} yields the $L^p$ expansion:
$$
h_{\varepsilon} = h_{0} + \varepsilon G_{0}(q) + o(\varepsilon),
$$
proving the $L^p$-differentiability at $\varepsilon = 0$.
\end{proof}

\section{A family of maps satisfying assumptions (A1)-(A9)}
A particular family of maps that verifies all hypotheses (A1)-(A9) is explicitly constructed here. Let $\varepsilon \in [0, \frac{1}{10})$ and 
\begin{equation}  \label{4.1}
T_{\varepsilon}(x)=\begin{cases}(1-\varepsilon)\left(\frac{3}{4}(2 x)+\frac{1}{4}(1-\sqrt[k]{1-2 x})\right)& \text {for } x \in\left[0, \frac{1}{2}\right),\\
 (1-\varepsilon)\left(\frac{3}{4}(2-2 x)+\frac{1}{4}(1-\sqrt[k]{2 x-1})\right) &\text{for } x \in\left(\frac{1}{2}, 1\right],
\end{cases}
\end{equation}
where $k > 2p$. In this example $c=\frac{1}{2}$ and it is immediate to see that $T_{\varepsilon}$ satisfies $(A1), (A2), (A3), (A5)$. Furthermore, 

$$
T_{\varepsilon}^{\prime}(x)= \begin{cases}(1-\varepsilon)\left(\frac{1}{2k(1-2 x)}(\sqrt[k]{1-2 x}+3k(1-2 x))\right) & \text { for } x \in\left[0, \frac{1}{2}\right), \\
(1-\varepsilon)\left(\frac{-1}{2k(2 x-1)}(\sqrt[k]{2 x-1}+3k(2 x-1))\right) & \text { for } x \in\left(\frac{1}{2}, 1\right],\end{cases}
$$
and $\left|T_{\varepsilon}^{\prime}(x)\right| \geqslant \frac{27k+9}{20k} > 1$ for $x \in\left[0, \frac{1}{2}\right) \cup\left(\frac{1}{2}, 1\right]$ and $\varepsilon \in\left[0, \frac{1}{10}\right)$. Thus, $T_{\varepsilon}$ satisfies $(A4)$. Also $T_{\varepsilon}^{\prime}$ satisfies $(A6)$ with $\beta=\frac{1-k}{k}$. Computing the second derivative we get 
$$
T_{\varepsilon}^{\prime \prime}(x)=\begin{cases}
\frac{(k-1) \cdot (1-\varepsilon)}{k^2 \cdot (2 x-1)^{2}} \cdot \sqrt[k]{1-2 x}  & \text { for } x \in\left[0, \frac{1}{2}\right), \\
\frac{(k-1) \cdot (1-\varepsilon)}{k^2 \cdot (2 x-1)^{2}}\cdot \sqrt[k]{2 x-1}  & \text { for } x \in\left(\frac{1}{2}, 1\right],
\end{cases}
$$
and so the first part of $(A7)$ is verified. Repeating the argument for the third derivative case provides the missing verification for $(A7)$ and $(A8)$. Now we will verify $(A9)$. Let $T_{0, \varepsilon}$ and $T_{1, \varepsilon}$ be the branches of $T_{\varepsilon}$, as in $(A1)$. We have $T_{i, \varepsilon}=D_{\varepsilon} \circ T_{i, 0}$, where $D_{\varepsilon}(x)=(1-\varepsilon)\cdot x$.

Let $P_{D_{\varepsilon}}, P_{T_{0}}$ denote the transfer operators associated with $D_{\varepsilon}$ and $T_{0}$, respectively. Note that have
$$
P_{T_{\varepsilon}}=P_{D_{\varepsilon}} \circ P_{T_{0}}
$$
and for any $x \in[0,1]$,
$$
P_{D_{\varepsilon}} g(x)=\left\{\begin{array}{c}
(1-\varepsilon)^{-1} g\left((1-\varepsilon)^{-1} x\right) \text { for } x \in[0, 1-\varepsilon] \\
0 \text { for } x \in(1-\varepsilon, 1] .
\end{array}\right.
$$

Firstly we will verify that
$$
\sup _{\|f\|_{W^{1,p}} \leqslant 1}\left\|\left(P_{T_{0}}-P_{T_{\varepsilon}}\right) f\right\|_{L^{2p}} \to 0
$$
as $\varepsilon \to 0$. Since we have already verified that $T_{\varepsilon}$ satisfies $(A 1), \ldots,(A 8)$, by Lemma \ref{lemma3.5} we obtain $g:=P_{T_{0}} f \in W^{1,p}$ when $f \in W^{1,p}$. Furthermore, there is an $M>0$ such that $\|f\|_{W^{1,p}} \leqslant 1$ implies $\|g\|_{W^{1,p}} \leqslant M$. Then, it is sufficient to prove that
\begin{equation} \label{4.2}
\sup _{\|g\|_{W^{1,p}} \leqslant M}\left\|g-P_{D_{\varepsilon}} g\right\|_{L^{2p}} \to 0 
\end{equation}
as $\varepsilon \to 0$. For brevity, we will also denote $P_{D_{\varepsilon}} g$ as
\begin{equation} \label{4.3}
P_{D_{\varepsilon}} g(x)=(1-\varepsilon)^{-1} 1_{[0,1-\varepsilon]} g\left((1-\varepsilon)^{-1} x\right).
\end{equation}

We can write by triangle inequality
\begin{equation} \label{4.4}
\|g - P_{D_\varepsilon} g\|_{L^{2p}} \leq \underbrace{\left(\int\limits_0^{1-\varepsilon} \left|g(x) - \frac{g\left(\frac{x}{1-\varepsilon}\right)}{1-\varepsilon}\right|^{2p} dx\right)^{\frac{1}{2p}}}_{\text{Term A}} + \underbrace{\left(\int\limits_{1-\varepsilon}^1 |g(x)|^{2p} dx\right)^{\frac{1}{2p}}}_{\text{Term B}}
\end{equation}

First, we will estimate \textbf{Term B} with $\varepsilon$. By the Sobolev embedding theorem, we have $\|g\|_{\infty} \leq C\|g\|_{W^{1,p}} \leq C \cdot M$ for some constant. Thus, $\sup_{x\in[1-\varepsilon,1]} |g(x)| \leq C \cdot M$ and 
\begin{align*}
      \textbf{Term B}\leq \left(\left(\sup_{x\in[1-\varepsilon,1]} |g(x)|\right)^{2p}\cdot\varepsilon\right)^{\frac{1}{2p}} \leq C \cdot M\cdot\varepsilon^{\frac{1}{2p}} \to 0 
\end{align*}

Now, let us analyze and evaluate \textbf{Term A} a little:

\begin{align*}
\left|g(x) - \frac{g\left(\frac{x}{1-\varepsilon}\right)}{1-\varepsilon}\right| 
&\leq \underbrace{\left|g(x) - g\left(\frac{x}{1-\varepsilon}\right)\right|}_{ \mathcal{S}_\varepsilon(x)} 
+ \underbrace{\left|1 - \frac{1}{1-\varepsilon}\right|\left|g\left(\frac{x}{1-\varepsilon}\right)\right|}_{ \mathcal{A}_\varepsilon(x)} = \mathcal{S}_\varepsilon(x) + \mathcal{A}_\varepsilon(x)
\end{align*}

Using the fundamental theorem of calculus and Hölder's inequality, we have:

\begin{align*}
    \mathcal{S}_\varepsilon(x) \leq \int_x^{\frac{x}{1-\varepsilon}} |g'(t)| dt \leq \left(\int_0^1 |g'|^p dt\right)^{\frac{1}{p}} \left(\tfrac{\varepsilon x}{1-\varepsilon}\right)^{1 - \frac{1}{p}} \leq M\left(\tfrac{\varepsilon x}{1-\varepsilon}\right)^{1 - \frac{1}{p}}
\end{align*}
and $\mathcal{A}_\varepsilon(x) = \tfrac{\varepsilon}{1-\varepsilon} \left|g\left(\tfrac{x}{1-\varepsilon}\right)\right|$.

Applying triangle inequality and variable substitution: 
\begin{align*}
\textbf{Term A} &\leq \|\mathcal{S}_\varepsilon\|_{L^{2p}} + \|\mathcal{A}_\varepsilon\|_{L^{2p}} \nonumber \\
&\leq M \left(\tfrac{\varepsilon}{1-\varepsilon}\right)^{1 - \frac{1}{p}} \left(\int_0^{1-\varepsilon} x^{2p(1 - \frac{1}{p})} dx\right)^{\frac{1}{2p}} + \tfrac{\varepsilon}{1-\varepsilon} \|g\|_{L^{2p}} \nonumber \\
&\leq C_1 M \varepsilon^{1 - \frac{1}{p}}  + C_2 M \varepsilon \to 0
\end{align*}
This proves that for $p>1$,  $\left\|g-P_{D_{\varepsilon}} g\right\|_{L^{2p}} \to 0$ as $\varepsilon \to 0$ when $\|f\|_{W^{1,p}} \leqslant 1$ and thus \eqref{2.1} is verified. 

Now for $f \in W^{2,p}$, we will verify the existence of the derivative $\left.\partial_{\varepsilon}P_{T_{\varepsilon}}f\right|_{\varepsilon=0}$ in $W^{1,p}$ and the linear response formula: 
\begin{equation}
    \left\|P_{T_{\varepsilon}}f - P_{T_{0}}f - \varepsilon\left(\left.\partial_{\varepsilon}P_{T_{\varepsilon}}f\right|_{\varepsilon=0}\right)\right\|_{W^{1,p}} = o(\varepsilon). 
\end{equation}
Since $f \in W^{2,p}$, then $g \in W^{2,p}$. Notice that we use $O(\varepsilon)$ in the $W^{2,p}$ sense. For small $\varepsilon$, we can expand $g\left(\frac{x}{1-\varepsilon}\right)$ around $\varepsilon=0$:
\begin{align*}
    g\left(\frac{x}{1-\varepsilon}\right) = g(x) + \varepsilon\cdot xg^{\prime}(x) + O(\varepsilon^{2}).
\end{align*}
Similarly, $\frac{1}{1-\varepsilon} = 1 + \varepsilon + O(\varepsilon^{2})$. Substituting these into $P_{D_{\varepsilon}}g$:
\begin{align*}
    P_{D_{\varepsilon}}g(x) = \left(1 + \varepsilon + O(\varepsilon^{2})\right)\left(g(x) + \varepsilon xg^{\prime}(x) + O(\varepsilon^{2})\right) \cdot 1_{[0,1-\varepsilon]}(x).
\end{align*}
Expanding and collecting terms up to $O(\varepsilon)$:
\begin{align*}
    P_{D_{\varepsilon}}g(x) = g(x) + \varepsilon\left(g(x) + xg^{\prime}(x)\right) + O(\varepsilon^{2}).
\end{align*}
Thus, the derivative of $P_{T_{\varepsilon}}f$ at $\varepsilon = 0$ is:
\begin{align*}
\partial_\varepsilon P_{T_z}f\big|_{\varepsilon=0} = \partial_\varepsilon (P_{D_\varepsilon}g)\big|_{\varepsilon=0} = g(x) + x g'(x) = P_{T_0}f(x) + x \cdot (P_{T_0}f)'(x).
\end{align*}
This exists in $W^{1,p}$ because $g \in W^{2,p}$ implies $g' \in W^{1,p}$, and multiplication by $x$ preserves $W^{1,p}$ regularity in $[0,1]$.
Now we estimate the remainder. The remainder term is:
\begin{align*}
    R_{\varepsilon}(x) = P_{T_{\varepsilon}}f(x) - P_{T_{0}}f(x) - \varepsilon\left( \partial_{\varepsilon}P_{T_{\varepsilon}}f\big|_{\varepsilon=0}\right) = O(\varepsilon^{2}).
\end{align*}
The $L^{p}$-norm of $R_{\varepsilon}$ is $O(\varepsilon^{2})$ by the definition. $W^{1,p}$ derivative of $R_{\varepsilon}$ is:
\begin{align*}
      R^{\prime}_{\varepsilon}(x) = \frac{d}{dx}\left(P_{D_{\varepsilon}}g(x) - g(x) - \varepsilon(g(x) + xg^{\prime}(x))\right).  
\end{align*}
Using the expansion:
\begin{align*}
    P_{D_{\varepsilon}}g(x) = g(x) + \varepsilon(g(x) + xg^{\prime}(x)) + \varepsilon^{2} \left(\text{higher-order terms}\right),
\end{align*}
we see that:
\begin{align*}
   R^{\prime}_{\varepsilon}(x) = O(\varepsilon^{2}) \cdot (\text{derivatives of } g, g^{\prime}, g^{\prime\prime}). 
\end{align*}
Since \(g \in W^{2,p}\), the derivatives are controlled, and:
\begin{align*}
    \left\|R_{\varepsilon}\right\|_{W^{1,p}} = O(\varepsilon^{2}) = o(\varepsilon).
\end{align*}
So, the error term is $o(\varepsilon)$ in $W^{1,p}$.

\section*{Acknowledgments} The author thanks Wael Bahsoun for proposing this problem.

\end{document}